\documentclass[11pt,lecno]{amsart}
\usepackage{indentfirst,amssymb,amsmath,amsthm}
\usepackage[mathscr]{euscript}
\usepackage[caption=false]{subfig}

\newtheorem{theorem}{Theorem}[section]

\newtheorem{lemma}{Lemma}[section]

\newtheorem{corollary}{Corollary}[section]

\newcommand{\be}{\begin{equation}}
\newcommand{\ee}{\end{equation}}
\newcommand{\bea}{\begin{eqnarray}}
\newcommand{\eea}{\end{eqnarray}}
\newcommand{\beas}{\begin{eqnarray*}}
\newcommand{\eeas}{\end{eqnarray*}}

\begin{document}
\setcounter{page}{1} \setlength{\unitlength}{1mm}\baselineskip
.58cm \pagenumbering{arabic} \numberwithin{equation}{section}

\title[Riemannian manifold]{Investigations on a Riemannian manifold with a semi-symmetric non-metric connection and gradient solitons}

\author[ K. De$^{*}$, U. C. De and A. Gezer ]
{ Krishnendu De $^{*}$, Uday Chand De and Aydin Gezer}

\address
 {$^{*}$ Department of Mathematics,
 Kabi Sukanta Mahavidyalaya,
The University of Burdwan.
Bhadreswar, P.O.-Angus, Hooghly,
Pin 712221, West Bengal, India. ORCID iD: https://orcid.org/0000-0001-6520-4520}
\email{krishnendu.de@outlook.in }
\address
{ Department of Pure Mathematics, University of Calcutta, West Bengal, India. ORCID iD: https://orcid.org/0000-0002-8990-4609}
\email {uc$_{-}$de@yahoo.com}

\address
{  Department of Mathematics, Ataturk University, Erzurum-TURKEY}
	\email {aydingzr@gmail.com}

\footnotetext {$\bf{2020\ Math.\ Subject\ Classification\:}.$ 53C05, 53C20, 53C25.
\\ {Key words: Riemannian manifolds, gradient Ricci solitons, gradient Yamabe solitons, gradient Einstein solitons, $m$-quasi Einstein solitons.\\
\thanks{$^{*}$ Corresponding author}
}}
\maketitle

\vspace{1cm}

\begin{abstract}
This article carries out the investigation of a three-dimensional Riemannian manifold $N^3$ endowed with a semi-symmetric type non-metric connection. Firstly, we construct a non-trivial example to prove the existence of a semi-symmetric type non-metric connection on $N^{3}$. It is established that a $N^3$ with the semi-symmetric type non-metric connection, whose metric is a gradient Ricci soliton, is a manifold of constant sectional curvature with respect to the semi-symmetric type non-metric connection. Moreover, we prove that if the Riemannian metric of $N^3$ with the semi-symmetric type non-metric connection is a gradient Yamabe soliton, then either $N^{3}$ is a manifold of constant scalar curvature or the gradient Yamabe soliton is trivial with respect to the semi-symmetric type non-metric connection. We also characterize the manifold $N^3$ with a semi-symmetric type non-metric connection whose metrics are Einstein solitons and $m$-quasi Einstein solitons of gradient type, respectively.

\end{abstract}

\maketitle

\section{Introduction}
In this paper, on a Riemannian manifold $N^3$, we carry out an investigation of gradient solitons with a semi-symmetric type non-metric connection (briefly, $SSNMC$). Many years ago, on a differentiable manifold, Friedman and Schouten \cite {fa} presented the concept of semi-symmetric linear connection. After that in 1932, on a Riemannian manifold, Hayden \cite {hha} introduced the notion of metric connection with torsion. In 1970, a systematic investigation of semi-symmetric metric connection which plays a significant role in the study of Riemannian manifolds was conducted by Yano \cite{ya1}. In this connection, we may mention the work of Zengin et al. (\cite{zen1}, \cite{zen2}).\par
On $N^3$, a linear connection $\widehat{\nabla}$ is named semi-symmetric if $\widehat{T}$, the torsion tensor defined by
\begin{equation*}
    \widehat{T}(U,V)=\widehat{\nabla}_{U}V-\widehat{\nabla}_{V}U-[U,V]
\end{equation*}
 obeys
\begin{equation}\label{a1} \widehat{T}(U,V)=\mathfrak{\psi} (V) U-\mathfrak{\psi} (U) V,\end{equation}
where $\mathfrak{\psi}$ is a $1$-form defined by $\mathfrak{\psi}(U)=g(U,\xi)$, for a fixed vector field $\xi$ (the associated vector field of $\widehat{\nabla}$). If in the right side of the equation (\ref{a1}) we substitute the independent vector fields $U$ and $V$, respectively, by $\phi U$ and $\phi V$, where $\phi$ is a (1,1)- tensor field\cite{go}, then the connection $\widehat{\nabla}$ transforms into a quarter-symmetric connection.\par

Again, if a semi-symmetric connection $\widehat{\nabla}$ on $N^3$ obeys
\begin{equation}\label{a2}(\widehat{\nabla}_{U}g)(V,Y) = 0,\end{equation}
then $\widehat{\nabla}$ is called metric \cite{ya1}. If $\widehat{\nabla}g \neq0$, then it is called non-metric \cite{hha}.
Here, we choose the $SSNMC$, that is, $\widehat{\nabla}g \neq0$ and the connection $\widehat{\nabla}$ obeys the equation (\ref{a1}). The concept of the $SSNMC$ on a Riemannian manifold was investigated in \cite{aga}. After that, several researchers investigated the properties of $SSNMC$ on manifolds with different structures (see, \cite{dea}, \cite{dog}, \cite{ozg}, \cite{sen}).\par

Hamilton \cite{rsh2} introduces the concept of Ricci flow as a solution to the challenge of obtaining a canonical metric on a smooth manifold. Ricci flow occurs when the metric of a Riemannian manifold $N^3$ is fulfilled by the evolution equation $\frac{\partial}{\partial t}g_{ij}(t)=-2\mathcal{S}_{ij}$, where $\mathcal{S}_{ij}$ and $g_{ij}$ are the components of the Ricci tensor and the metric tensor, respectively. Ricci solitons were created via self-similar solutions to the Ricci flow.\par

A metric of $N^3$ is named a Ricci soliton \cite{rsh1} if it fulfills
\begin{equation}
\label{4}
\mathfrak{L}_{W}g +2\lambda g+2\widehat{\mathcal{S}}=0,
\end{equation}
for some $\lambda \in \mathbb{R}$, the set of real numbers. Here, $\mathfrak{L}$ being the Lie derivative operator and $\widehat{\mathcal{S}}$ is the Ricci tensor with respect to the non-metric connection $\widehat{\nabla}$. $W$ is a complete vector field known as a potential vector field. The Ricci soliton is considered to be shrinking, expanding or steady depending on whether $\lambda$ is negative, positive, or zero. If $W$ is Killing or zero, the Ricci soliton is trivial and $N^3$ is Einstein. Also, if $W=Df$ for some smooth function $f$, then equation (\ref{4}) turns into
\begin{equation}
\label{a4}
\widehat{\nabla}^2 \, f+\widehat{\mathcal{S}}+\lambda g=0,
\end{equation}
where $\widehat{\nabla}^2$ and $D$ indicate the Hessian and the gradient operator of $g$, respectively. The metric obeying the equation (\ref{a4}) is called a gradient Ricci soliton. Here, $f$ is said to be the potential function of the gradient Ricci soliton.\par

On a complete Riemannian manifold $N^3$, Hamilton \cite{rsh2} proposed the idea of Yamabe flow, which was inspired by Yamabe's conjecture ("metric of a complete Riemannian manifold is conformally connected to a metric with constant scalar curvature").
A Riemannian manifold $N^3$ equipped with a Riemannian metric $g$ is called a Yamabe flow if it obeys:
\begin{equation*}
\frac{\partial}{\partial t}g(t)+r g(t)=0, \,\,\,\,\, g_{0}=g(t),
\end{equation*}
where $t$ indicates the time and $r$ being the scalar curvature of $N^3$. A Riemannian manifold $N^3$ equipped with a Riemannian metric $g$ is named a Yamabe soliton if it fulfills
\begin{equation}
\label{a5}
 \mathfrak{L}_{W}g-2(\widehat{r}-\lambda)g=0
\end{equation}
for real constant $\lambda : M \rightarrow \mathbb{R}$ and $\widehat{r}$  is the scalar curvature with respect to the non-metric connection $\widehat{\nabla}$. Here, $W$ is called the potential vector field. In $N^3$, with the condition $W=Df$, the Yamabe soliton reduces to the gradient Yamabe soliton. Thus, (\ref{a5}) takes the form
\begin{equation}
\label{a6}
 \widehat{\nabla}^2 f-(\widehat{r}-\lambda)g=0.
\end{equation}
If $f$ is constant (or, $W$ is Killing) on $M$, then the soliton becomes trivial. The $3$-Kenmotsu manifolds and almost co-K\"{a}hler manifolds with Yamabe solitons have been characterized by Wang \cite{yw1} and Suh and De \cite{suh2}, respectively. Chen and Deshmukh (\cite{chen1}, \cite{chen2}) studied the Yamabe solitons on Riemannian manifolds. Some interesting outcomes on this solitons have been investigated in ( \cite{blaga2}, \cite{ozg1} \cite{kde1}, \cite{kde2}, \cite{hsu}) and also by others.\\

The notion of gradient Einstein soliton was presented by Catino and Mazzieri \cite{cm} and obeys
\begin{equation}\label{a7}
    \widehat{\mathcal{S}}-\frac{1}{2}\widehat{r}g+\widehat{\nabla}^{2}f+\lambda g=0,
\end{equation}
where $\lambda \in$ $\mathbb{R}$ is a constant and $f$ indicates a smooth function.\par

A Riemannian manifold $N^3$ endowed with the Riemannian metric $g$ is named a gradient m-quasi Einstein metric \cite{br1} if there exists a constant $\lambda$, a smooth function $f:N^{3}\rightarrow \mathbb{R}$ and obeys
\begin{equation} \widehat{\mathcal{S}}-\lambda g+\widehat{\nabla}^2 \, f -\frac{1}{m}df \otimes df =0,\label{a8}\end{equation}
where $\otimes$ indicate the tensor product and $m$ is an integer. In this case $f$ being the $m$-quasi Einstein potential function \cite{br1}. Here, the gradient m-quasi Einstein soliton is expanding for $\lambda > 0$, steady for $\lambda = 0$ and shrinking when $\lambda < 0$. If $m = \infty$, the foregoing equation represents a gradient Ricci soliton and the metric represents almost gradient Ricci soliton if it obeys the condition $m = \infty$ and $\lambda$ is a smooth function. Few characterizations of the above metrics were characterized by He et al. \cite{he}.\par

The foregoing investigations motivate us to study the Riemannian manifold $N^3$ endowed with a $SSNMC$.\par

The content of the paper is laid out as:
In Section $2$, we produce the preliminary ideas of $SSNMC$. The existence of a $SSNMC$ on a Riemannian manifold are established in Section $3$. The gradient Ricci soliton on $N^3$ equipped with a $SSNMC$ is investigated in Section $4$. Section $5$ concerns with gradient Yamabe soliton on $N^3$ with a $SSNMC$. We study the properties of $N^3$ with a $SSNMC$ whose metrics are gradient Einstein solitons and gradient $m$-quasi Einstein solitons, in Section $6$ and Section $7$, respectively.

\section{Semi-symmetric non-metric connection}
A linear connection $\widehat{\nabla}$ on $N$, defined by
\begin{equation}
\widehat{\nabla}_{U}V = \nabla_{U}V + \mathfrak{\psi}(V)U,\label{b1}
\end{equation}
$\nabla$ being the Levi-Civita connection, is a $SSNMC$.
It also obeys
\begin{equation}
 (\widehat{\nabla}_{U}g)(V,Y)=-\mathfrak{\psi}(V)g(U,Y) - \mathfrak{\psi}(Y)g(U,V).\label{b2}
\end{equation}
 Then $\widehat{R}$, the  curvature  tensor with  respect  to the $SSNMC$, $\widehat{\nabla}$,  and  $R$, the Riemannian curvature tensor are  related  by  \cite{aga}
\begin{equation}
\widehat{R}(U,V)Y = R(U,V)Y - \alpha^{*}(V,Y)U + \alpha^{*}(U,Y)V,\label{b3}
\end{equation}
for all $U,V,Y$ on $N^3$, where $\alpha^{*}$ is a $(0,2)$- tensor field defined by
\begin{equation}
\alpha^{*}(U,V)= (\nabla_{U}\mathfrak{\psi})(V) - \mathfrak{\psi} (U)\mathfrak{\psi}(V).\label{b4}
\end{equation}
Throughout this article, we choose that the vector field $\xi$ is a unit parallel vector field with respect to the Levi-Civita connection $\nabla$. Then $\nabla_{U} \xi=0$, which immediately implies
\begin{equation}\label{b5}
    R(U,V)\xi=0
\end{equation}
and
\begin{equation}\label{b6}
    \mathcal{S}(U,\xi)=0.
\end{equation}
Also, using $\nabla_{U} \xi=0$, we obtain
\begin{equation}\label{b7}
    (\nabla_{U} \mathfrak{\psi})V=0.
\end{equation}
Hence by the preceding equation, we get from $(\ref{b3})$
\begin{equation}\label{b8}
\widehat{R}(U,V)Y = R(U,V)Y + \mathfrak{\psi}(Y)[ \mathfrak{\psi}(V)U-\mathfrak{\psi}(U)V],
\end{equation}
From the foregoing equation, we can easily have
\begin{equation}\label{b9}
    \widehat{\mathcal{S}}(U,V)=\mathcal{S}(U,V)+2\mathfrak{\psi}(U)\mathfrak{\psi}(V).
\end{equation}
Contracting the above equation, we lead
\begin{equation}\label{b10}
    \widehat{r}=r-2,
\end{equation}
since $\mathfrak{\psi}(\xi)=g(\xi,\xi)=1$.
Making use of (\ref{b5}), we infer from (\ref{b8})
\begin{equation}\label{b11}
\widehat{R}(U,V)\xi =  \mathfrak{\psi}(V)U-\mathfrak{\psi}(U)V.
\end{equation}
Therefore, we obtain the subsequent relations
\begin{equation}\label{b12}
    \mathfrak{\psi}(\widehat{R}(U,V)Y)=0,
\end{equation}
\begin{equation}\label{b13}
    \widehat{\mathcal{S}}(U,\xi)=2 \mathfrak{\psi}(U),\,\,\,\,\widehat{Q}\xi=2\xi.
\end{equation}

We first establish the subsequent Lemma:
\begin{lemma}
Let $N^3$ be a Riemannian manifold with a $SSNMC$, $\widehat{\nabla}$. Then we have
\begin{equation}\label{b14}
\xi \widehat{r}=0.
\end{equation}
\end{lemma}

\begin{proof}
In $N^3$, the Riemannian curvature tensor is expressed by
\begin{eqnarray}\label{b15}
R(U,V)Y &=&\; g(V,Y)Q U-g(U,Y)Q V\; +\mathcal{S}(V,Y)U-\mathcal{S}(U,Y)V \nonumber\\
 && \;-\frac{r}{2}[g(V,Y)U -g(U,Y)V],
\end{eqnarray}
Making use of (\ref{b8}) and (\ref{b9}), we acquire
\begin{eqnarray}\label{b16}
\widehat{R}(U,V)Y&-&\mathfrak{\psi}(Y)[\mathfrak{\psi}(V)U-\mathfrak{\psi}(U)V]= g(V,Y)[\widehat{Q} U-2\xi \mathfrak{\psi}(U)]\nonumber\\&&
-g(U,Y)[\widehat{Q} V-2\xi \mathfrak{\psi}(V)]+[\widehat{\mathcal{S}}(V,Y)-2\mathfrak{\psi}(V)\mathfrak{\psi}(Y)]U\nonumber\\&&
-[\widehat{\mathcal{S}}(U,Y)-2\mathfrak{\psi}(U)\mathfrak{\psi}(Y)]V\nonumber\\&&
-\frac{r}{2}[g(V,Y)U -g(U,Y)V],
\end{eqnarray}
Putting $V=Y=\xi$, the foregoing equation yields
\begin{equation}\label{b17}
\widehat{Q}U=(\frac{\widehat{r}}{2}+1)U-(\frac{\widehat{r}}{2}-1)\mathfrak{\psi}(U)\xi.
\end{equation}
Taking covariant derivative along $V$, we write
\begin{equation}\label{b18}
(\nabla_{V}\widehat{Q})U=\frac{(V \widehat{r})}{2}[U-\mathfrak{\psi}(U)\xi] .
\end{equation}
Contracting the foregoing equation we acquire the desired result.
\end{proof}

The projective curvature tensor $\widehat{P}$ of $N^3$ with respect to $\widehat{\nabla}$ is defined by
\begin{equation}\label{b19}
\widehat{P}(U,V)Y = \widehat{R}(U,V)Y -
\frac{1}{2}[\widehat{\mathcal{S}}(V,Y)U - \widehat{\mathcal{S}}(U,Y)V].
\end{equation}

Making use of ($\ref{b8}$) and ($\ref{b9}$), ($\ref{b19}$) reduces to
\begin{equation}\label{b20}\widehat{P}(U,V)Y = P(U,V)Y , \end{equation} where $P$ represents the projective curvature tensor with respect to the Levi-Civita connection
$\nabla$ defined by
\begin{equation*}
    P(U,V)Y = R(U,V)Y -\frac{1}{2}[\mathcal{S}(V,Y)U - \mathcal{S}(U,Y)V].
\end{equation*}

\begin{theorem}
\label{tb1}
If $N^{3}$ is endowed with a $SSNMC$ $\widehat{\nabla}$, then the projective curvature tensor with respect to $\widehat{\nabla}$ and $\nabla$, respectively, coincide on $N^3$.
\end{theorem}

In differential geometry, the investigation of conformal curvature tensor performs a significant role. Also, it has various applications in applied physics and the other branches of modern sciences. Motivated by the above facts we investigate the properties of the conformal curvature tensor $C$.
With respect to $\widehat{\nabla}$, the conformal curvature tensor $\widehat{C}$ is defined
by
\begin{eqnarray}\label{b21}
 \widehat{C}(U,V)Y&=&\widehat{R}(U,V)Y -[\widehat{\mathcal{S}}(V,Y)U - \widehat{\mathcal{S}}(U,Y)V+g(V,Y)\widehat{Q}U\nonumber\\&&-g(U,Y)\widehat{Q}V]
 +\frac{\widehat{r}}{2}[g(V,Y)U-g(U,Y)V],
\end{eqnarray}
for all $U$, $V$ and $Y$ on $N^3$ \cite{wei}.
Utilizing ($\ref{b8}$) and ($\ref{b9}$) in ($\ref{b21}$), we obtain
\begin{eqnarray}\label{b22}
 \widehat{C}(U,V)Y&=&C(U,V)Y -\widehat{\psi}(V)\widehat{\psi}(Y)U+\widehat{\psi}(U)\widehat{\psi}(Y)V\nonumber\\&&
+2\xi g(V,Y)\widehat{\psi}(U)-2\xi g(U,Y)\widehat{\psi}(V)\nonumber\\&&
+g(V,Y)U-g(U,Y)V,
\end{eqnarray}
where $C$ represents the conformal curvature tensor with respect to the Levi-Civita connection $\nabla$ defined by \begin{eqnarray}\nonumber
 C(U,V)Y&=&R(U,V)Y -[\mathcal{S}(V,Y)U - \mathcal{S}(U,Y)V+g(V,Y)QU\nonumber\\&&-g(U,Y)QV]
 +\frac{r}{2}[g(V,Y)U-g(U,Y)V].
\end{eqnarray}

Putting $Y=\xi$ in (\ref{b22}), we get

\begin{equation}\label{b23}
\widehat{C}(U,V)\xi = C(U,V)\xi.\end{equation}
Hence, we have the subsequent theorem:
\begin{theorem}
\label{tb2}
If $N^{3}$ is equipped with a $SSNMC$ $\widehat{\nabla}$, then the the conformal curvature tensor with respect to $\widehat{\nabla}$ and $\nabla$, satisfy the relation (\ref{b23}).
\end{theorem}
\section{Existence of a semi-symmetric type non-metric connection}
Here we construct a non-trivial example of semi-symmetric type non-metric connection on a Riemannian manifold.

\subsection{Example 1:}

Let us consider a three-dimensional differentiable manifold $N^3=\{(u,v,w)\in \mathbb{R}^{3}, w\neq0\},$ where
$(u,v,w)$ indicates the standard coordinate of $\mathbb{R}^{3}.$
Let us choose
\begin{equation*}
    k_{1}=e^w\frac{\partial }{\partial
u},\hspace{7pt}k_{2}=e^w\frac{\partial } {\partial v}
,\hspace{7pt}k_{3}=\frac{\partial }{\partial w}.
\end{equation*}
At each point of $N^3$ the preceding vector fields are linearly independent.
Here we define the Riemannian metric $g$ as
\begin{equation*}
    g(k_{1},k_{3})
    =g(k_{1},k_{2})=g(k_{2},k_{3})=0,
\end{equation*}
\begin{equation*}
    g(k_{1},k_{1})
    =g(k_{2},k_{2})=g(k_{3},k_{3})=1.
\end{equation*}
$\mathfrak{\psi}$ indicates a $1$-form defined by $\mathfrak{\psi}(U)=g(U,\xi)$, where $\xi=k_{3}$.
Hence, $(N^3 ,g)$ is a three-dimensional Riemannian manifold.
 The Lie brackets are calculated as
\begin{eqnarray*}
[k_{1},k_{3}]
&=&k_{1}k_{3}-k_{3}k_{1}\nonumber\\
&=&e^w \frac{\partial}{\partial u}
(\frac{\partial }{\partial w})-(\frac{\partial }{\partial w})(e^w \frac{\partial }{\partial u})\nonumber\\
&=&e^w \frac{\partial ^{2}}{\partial u\partial
w}-e^w \frac{\partial ^{2}}{\partial w\partial u}-e^w \frac{\partial }{\partial u}\nonumber\\&=&-k_{1}.
\end{eqnarray*}

Similarly,
\begin{equation*}
    [k_{1},k_{2}]=0\hspace{10pt}and\hspace{10pt}
[k_{2},k_{3}]=-k_{2}.
\end{equation*}

$\nabla $, the Levi-Civita connection with respect to $g$ is obtained by
\begin{eqnarray}
2g(\nabla
_{U}V,Y)&=&U g(V,Y)+V g(Y,U)-Y g(U,V)\nonumber\\&-&g(U,[V,Y])-g(V,[U,Y])
+g(Y,[U,V]),\label{f1}
\end{eqnarray}
 which is termed as Koszul's formula.

Making use of (\ref{f1}) we have
\begin{eqnarray}
2g(\nabla_{k_{1}}k_{3},k_{1})
&=&-2g(k_{1},k_{1}).\label{f2}
\end{eqnarray}

Again by (\ref{f1})
\begin{equation}\label{f3}
    2g(\nabla_{k_{1}}k_{3},k_{2})=0=-2g(k_{1},k_{2})
\end{equation}
and
\begin{equation}\label{f4}
    2g(\nabla _{k_{1}}k_{3},k_{3})=0=-2g(k_{1},k_{3}).
\end{equation}

From (\ref{f2}), (\ref{f3}) and (\ref{f4}) we get
$$2g(\nabla _{k_{1}}k_{3},U)=-2g(k_{1},U),$$ for all $U\in \mathfrak{X}(N).$

Thus, $$\nabla _{k_{1}}k_{3}=-k_{1}.$$

Therefore, (\ref{f1}) further gives
$$\nabla _{k_{1}}k_{2}=0,\hspace{10pt}
\nabla _{k_{1}}k_{1}=k_{3},$$
$$\nabla _{k_{2}}k_{3}=-k_{2},\hspace{10pt}\nabla _{k_{2}}k_{2}=k_{3},\hspace{10pt}
\nabla _{k_{2}}k_{1}=0,$$
\begin{equation}\label{f5}
    \nabla
_{k_{3}}k_{3}=0,\hspace{10pt}\nabla _{k_{3}}k_{2}=0,\hspace{10pt}
\nabla _{k_{3}}k_{1}=0.
\end{equation}
We know that
\begin{equation}\label{f6}
    R(U,V)Y=\nabla _{U}\nabla _{V}Y-\nabla _{V}\nabla
_{U}Y-\nabla _{[U,V]}Y,
\end{equation}
 where $R$ is the Riemann curvature tensor.
Utilizing the foregoing results and with the help of (\ref{f6}), we acquire
$$R(k_{1},k_{2})k_{3}=0,\hspace{10pt}
R(k_{2},k_{3})k_{3}=-k_{2},\hspace{10pt}
R(k_{1},k_{3})k_{3}=-k_{1},$$
$$R(k_{1},k_{2})k_{2}=-k_{1},\hspace{10pt}
R(k_{2},k_{3})k_{2}=k_{3},\hspace{10pt}
R(k_{1},k_{3})k_{2}=0,$$
$$R(k_{1},k_{2})k_{1}=k_{2},\hspace{10pt}
R(k_{2},k_{3})k_{1}=0,\hspace{10pt}
R(k_{1},k_{3})k_{1}=k_{3}.$$\\

Using the above expressions, the Ricci tensor can be obtained as
\begin{eqnarray*}
\mathcal{S}(k_{1},k_{1})&=&g(R(k_{1},k_{2})k_{2},k_{1})
+g(R(k_{1},k_{3})k_{3},k_{1})
\nonumber\\&=&-2.
\end{eqnarray*}

Similarly, we get
\begin{equation*}
    \mathcal{S}(k_{2},k_{2})=\mathcal{S}(k_{3},k_{3})=-2.
\end{equation*}

Therefore, the scalar curvature $r$ is calculated as
\begin{equation*}
    r=\mathcal{S}(k_{1},k_{1})+\mathcal{S}(k_{2},k_{2})
    +\mathcal{S}(k_{3},k_{3})=-6.
\end{equation*}
Making use of the above expressions and using the equation (\ref{b1}), we have

$$\widehat{\nabla} _{k_{1}}k_{3}=0\hspace{10pt}
\widehat{\nabla} _{k_{1}}k_{2}=0,\hspace{10pt}
\widehat{\nabla} _{k_{1}}k_{1}=k_{3},$$
$$\widehat{\nabla} _{k_{2}}k_{3}=0,\hspace{10pt}
\widehat{\nabla} _{k_{2}}k_{2}=k_{3},\hspace{10pt}
\widehat{\nabla} _{k_{2}}k_{1}=0,$$
\begin{equation}\label{f7}
    \widehat{\nabla}_{k_{3}}k_{3}=k_3,\hspace{10pt}
    \widehat{\nabla} _{k_{3}}k_{2}=0,\hspace{10pt}
\widehat{\nabla} _{k_{3}}k_{1}=0.
\end{equation}
From the last equation and using (\ref{a1}), we obtain $\widehat{T}(k_{1},k_{3})=k_{1}$ and $\mathfrak{\psi} (k_{3}) k_{1}-\mathfrak{\psi} (k_{1}) k_{3}=k_{1}$. Similarly, other components can be verified. Therefore, the linear connection $\widehat{\nabla}$ defined on ($N^3, g$) as $(\ref{b1})$, is a semi-symmetric connection. Also, we have
\begin{equation*}
    (\widehat{\nabla}_{k_{1}}g)(k_{1},k_{3})=-1 \neq 0.
\end{equation*}
Thus, the linear connection $\widehat{\nabla}$ is non-metric on ($N^3, g$).

\section{ gradient Ricci solitons on $N^3$ with a $SSNMC$}

This section carries out the study of gradient Ricci solitons in $N^3$ with a $SSNMC$.\par

Let us choose that the soliton vector $W$ of the Ricci soliton $(g, W, \lambda)$ in $N^3$ with a $SSNMC$ is a gradient of some smooth function $f$. Then using (\ref{a4}), we infer
\begin{equation}
\label{c1}
\widehat{\nabla}_{U}Df=-\widehat{Q}U-\lambda U,
\end{equation}
for all $U \in \mathfrak{X}(N)$.
Making use of the above equation, the subsequent relation
\begin{equation}
\label{c2}
\widehat{R}(U, V)Df=\widehat{\nabla}_{U} \widehat{\nabla}_{V}Df-\widehat{\nabla}_{V} \widehat{\nabla}_{U}Df-\widehat{\nabla}_{[U, V]}Df
\end{equation}
yields
\begin{equation}
\label{c3}
\widehat{R}(U, V)Df=(\widehat{\nabla}_{U}\widehat{Q})(V)-(\widehat{\nabla}_{V}\widehat{Q})(U).
\end{equation}
The contraction of the preceding equation gives
\begin{equation}
\label{c4}
\widehat{\mathcal{S}}(U, Df)=-\frac{1}{2}(U\widehat{r}).
\end{equation}
Again, from (\ref{b17}) we obtain
\begin{equation}
\label{c5}
\widehat{\mathcal{S}}(U, Df)=(\frac{\widehat{r}}{2}+1)(U f)-(\frac{\widehat{r}}{2}-1)\mathfrak{\psi}(U)(\xi f).
\end{equation}
Comparing the equations (\ref{c4}) and (\ref{c5})
\begin{equation}\label{c6}
 -\frac{1}{2}(U\widehat{r})=(\frac{\widehat{r}}{2}+1)(U f)-(\frac{\widehat{r}}{2}-1)\mathfrak{\psi}(U)(\xi f).
\end{equation}
Now, putting $U=\xi$ in (\ref{c6}), we find
\begin{equation}
\label{c7}
\xi f=0,
\end{equation}
since $\xi \widehat{r}=0$.

Equation (\ref{c3}) gives
\begin{equation}\label{c9}
    g(\widehat{R}(U,V)\xi, Df)=0.
\end{equation}
Again, from equation (\ref{b11}) we infer that
\begin{equation}\label{c10}
    g(\widehat{R}(U,V)\xi, Df)=\mathfrak{\psi}(V)(U f)-\mathfrak{\psi}(U)(V f).
\end{equation}
Comparing last two equations and putting $V=\xi$ and using $\xi f=0$, we lead
\begin{equation}\label{c11}
    U f=0,
\end{equation}
which shows that $f=$ constant. Making use of the fact that $f$ is constant, equation (\ref{c1}) infers that the manifold is an Einstein manifold. Hence, the Riemannian manifold $N^3$ is of constant sectional curvature.

\begin{theorem}
\label{tc1}
Let the soliton vector field $W$ of the Ricci soliton $(g, W, \lambda)$ in $N^3$ with a $SSNMC$ be a gradient Ricci soliton. Then  $N^{3}$ is a manifold of constant sectional curvature with respect to the $SSNMC$.
\end{theorem}

\section{Gradient Yamabe solitons on $N^3$ with a $SSNMC$}

From equation (\ref{a6}), we find
\begin{equation}
\label{4.1}
\widehat{\nabla}_{V}Df=(\widehat{r}-\lambda)V.
\end{equation}
Differentiating (\ref{4.1}) covariantly along the vector field $U$, we obtain
\begin{equation}
\label{4.2}
\widehat{\nabla}_{U}\widehat{\nabla}_{V}Df=(U \widehat{r})V+(\widehat{r}-\lambda)\widehat{\nabla}_{U}V.
\end{equation}
Interchanging $U$ and $V$ in the above equation and then utilizing the preceding equation in $\widehat{R}(U, V)Df=\widehat{\nabla}_{U} \widehat{\nabla}_{V}Df-\widehat{\nabla}_{V} \widehat{\nabla}_{U}Df-\widehat{\nabla}_{[U, V]}Df$, we lead
\begin{equation}
\label{4.3}
\widehat{R}(U, V)Df=(U \widehat{r})V-(V \widehat{r})U.
\end{equation}
Contracting the previous equation over $U$, we get
\begin{equation}
\label{4.4}
\widehat{\mathcal{S}}(V, Df)=-2(V \widehat{r}).
\end{equation}
Combining the last equation and (\ref{c5}), we infer
\begin{equation}
\label{4.5}
-2(U\widehat{r})=(\frac{\widehat{r}}{2}+1)(U f)-(\frac{\widehat{r}}{2}-1)\mathfrak{\psi}(U)(\xi f).
\end{equation}
Putting $U=\xi$ in the foregoing equation, we have
\begin{equation}
\label{4.6}
\xi f=0,
\end{equation}
since $\xi \widehat{r}=0$.
Thus, from (\ref{4.5}), we obtain
\begin{equation}\label{4.7}
 -2(U\widehat{r})=(\frac{\widehat{r}}{2}+1)(U f).
\end{equation}
Now, from equation (\ref{4.3}) we find that
\begin{equation}\label{4.8}
    g(\widehat{R}(U,V)\xi, Df)=\mathfrak{\psi}(U)(V \widehat{r})-\mathfrak{\psi}(V)(U \widehat{r}).
\end{equation}
Combining equation (\ref{c10}) and (\ref{4.8}), we have
\begin{equation}\label{4.9}
\mathfrak{\psi}(V)(U f)-\mathfrak{\psi}(U)(V f)=\mathfrak{\psi}(U)(V \widehat{r})-\mathfrak{\psi}(V)(U \widehat{r}).
\end{equation}
Setting $V=\xi$ in the previous equation gives
\begin{equation}\label{4.10}
(U \widehat{r})=-(U f).
\end{equation}
Utilizing (\ref{4.10}) in (\ref{4.7}) we infer that
\begin{equation}\label{4.11}
(\frac{\widehat{r}}{2}-1)(U f)=0,
\end{equation}
which entails that either $\widehat{r} =2$ or $\widehat{r} \ne 2$.\\

If $\widehat{r} =2$, then from (\ref{b10}) we infer that $r=4$. Therefore, $N^3$ is of constant scalar curvature.\par
Next, we suppose that $\widehat{r} \ne 2$, that is,  $(U f)=0$, which implies $f$ is a constant. Therefore, the gradient Yamabe soliton is trivial.\\
Hence, we state the result as:
\begin{theorem}
\label{thm4.1}
Let the Riemannian metric of $N^3$ with a $SSNMC$ be the gradient Yamabe soliton. Then, either $N^{3}$ is a manifold of constant scalar curvature or the gradient Yamabe soliton is trivial with respect to the $SSNMC$.
\end{theorem}
Also, if $\widehat{r} =2$, then using the equation (\ref{b17}) we acquires that the manifold is an Einstein manifold. Hence, the Riemannian manifold $N^3$ is of constant sectional curvature.

\begin{corollary}
\label{tccc1}
Let the Riemannian metric of $N^3$ with a $SSNMC$ be the gradient Yamabe soliton. Then, either $N^{3}$ is a manifold of constant sectional curvature or the gradient Yamabe soliton is trivial with respect to the $SSNMC$.
\end{corollary}

\section{Gradient Einstein solitons on $N^3$ with a $SSNMC$}

Making use of (\ref{a7}), we have
\begin{equation}
\label{5.1}
\widehat{\nabla}_{V}Df=-\widehat{Q}V+\frac{\widehat{r}}{2}V-\lambda V.
\end{equation}
Differentiating (\ref{5.1}) covariantly along $U$, we find
\begin{equation}
\label{5.2}
\widehat{\nabla}_{U}\widehat{\nabla}_{V}Df=-\widehat{\nabla}_{U} \widehat{Q}V+\frac{1}{2}(U \widehat{r})V+(\frac{\widehat{r}}{2}-\lambda)\widehat{\nabla}_{U}V.
\end{equation}
Interchanging $U$ and $V$ and then making use of the above equation in $\widehat{R}(U, V)Df=\widehat{\nabla}_{U} \widehat{\nabla}_{V}Df-\widehat{\nabla}_{V} \widehat{\nabla}_{U}Df-\widehat{\nabla}_{[U, V]}Df$, we infer
\begin{eqnarray}
\label{5.3}
\widehat{R}(U, V)Df&=&\frac{1}{2}[(U \widehat{r})V-(V \widehat{r})U]\nonumber\\&&
-(\widehat{\nabla}_{U}\widehat{Q})(V)+(\widehat{\nabla}_{V}\widehat{Q})(U).
\end{eqnarray}
Contracting the foregoing equation over $U$, we obtain
\begin{equation}
\label{5.4}
\widehat{\mathcal{S}}(V, Df)=-\frac{1}{2}(V \widehat{r}).
\end{equation}
Combining the last equation and (\ref{c5}), we get
\begin{equation}
\label{5.5}
-\frac{1}{2}(U\widehat{r})=(\frac{\widehat{r}}{2}+1)(U f)-(\frac{\widehat{r}}{2}-1)\mathfrak{\psi}(U)(\xi f).
\end{equation}
Setting $U=\xi$ in (\ref{5.5}), we have
\begin{equation}
\label{5.6}
(\xi f)=0,
\end{equation}
since $\xi \widehat{r}=0$.
Thus, from (\ref{5.5}), we acquire
\begin{equation}\label{5.7}
 -\frac{1}{2}(U\widehat{r})=(\frac{\widehat{r}}{2}+1)(U f).
\end{equation}
Now, from equation (\ref{5.3}) we obtain that
\begin{equation}\label{5.8}
    g(\widehat{R}(U,V)\xi, Df)=-\frac{1}{2}[\mathfrak{\psi}(U)(V \widehat{r})-\mathfrak{\psi}(V)(U \widehat{r})].
\end{equation}
Combining equation (\ref{c10}) and (\ref{5.8}), we lead
\begin{equation}\label{5.9}
\mathfrak{\psi}(V)(U f)-\mathfrak{\psi}(U)(V f)=-\frac{1}{2}[\mathfrak{\psi}(U)(V \widehat{r})-\mathfrak{\psi}(V)(U \widehat{r})].
\end{equation}
Putting $V=\xi$ in the last equation yields
\begin{equation}\label{5.10}
(U f)=-\frac{1}{2}(U \widehat{r}).
\end{equation}
Using (\ref{5.10}) in (\ref{5.7}) we find that
\begin{equation}\label{5.11}
\frac{\widehat{r}}{2}(U f)=0.
\end{equation}
Hence, either $\widehat{r} =0$ or $\widehat{r} \ne 0$.\par
If $\widehat{r} =0$, then from (\ref{b10}) we acquire that $r=2$. Therefore, $N^3$ is of constant scalar curvature.\par
Next, we suppose that $\widehat{r} \ne 0$, that is, $(U f)=0$, which implies $f$ is a constant. Then, equation (\ref{5.1}) reveals that $N^3$ is an Einstein manifold. Hence, $N^3$ is of constant sectional curvature, since the manifold is of dimension $3$,.\par
Thus, we state the subsequent:
\begin{theorem}
\label{thm5.1}
 If the Riemannian metric of $N^3$ with a $SSNMC$ is a gradient Einstein soliton, then $N^{3}$ is either a manifold of constant scalar curvature or a manifold of constant sectional curvature with respect to the $SSNMC$.
\end{theorem}

\section{Gradient $m$-quasi Einstein solitons on $N^3$ with a $SSNMC$}
Here, we investigate the $m$-quasi Einstein metric on $N^3$ with a $SSNMC$. Initially, we prove the following Lemma
\begin{lemma}\label{lem1}
In $N^3$, we have the following:
\begin{eqnarray}\label{6.1}
\widehat{R}(U,V)D f &=& (\widehat{\nabla}_{V}\widehat{Q})U-(\widehat{\nabla}_{U} \widehat{Q})V+\frac{\lambda}{m}\{ (V f)U-(U f)V\}\nonumber\\&&
+\frac{1}{m}\{ (U f)\widehat{Q} V-(V f)\widehat{Q} U \},
\end{eqnarray}
for all $U, \, V \in \mathfrak{X}(M)$.
\end{lemma}
\begin{proof}
Let the Riemannian metric of $N^3$ with a $SSNMC$ be a $m$-quasi Einstein metric. Therefore, the equation (\ref{a8}) can be represented as
\begin{equation}\label{6.2}
\widehat{\nabla}_{U}D f=-\widehat{Q} U+\frac{1}{m}g(U,D f)D f+\lambda U.
\end{equation}
Covariant derivative of (\ref{6.2}) along $V$ yields
\begin{eqnarray}\label{6.3}
\widehat{\nabla}_{V}\widehat{\nabla}_{U}D f &=& -\widehat{\nabla}_{V}\widehat{Q} U+ \frac{1}{m} \widehat{\nabla}_{V}g(U,D f)D f\nonumber\\&& +\frac{1}{m} g(U,D f)\widehat{\nabla}_{V}D f+ \lambda \widehat{\nabla}_{V}U.
\end{eqnarray}
Exchanging $U$ and $V$ in (\ref{6.3}), we obtain
\begin{eqnarray}\label{6.4}
\widehat{\nabla}_{U}\widehat{\nabla}_{V}D f &=& -\widehat{\nabla}_{U}\widehat{Q}  V+ \frac{1}{m}\widehat{\nabla}_{U}g(V,D f)D f\nonumber\\&& +\frac{1}{m} g(V,D f)\widehat{\nabla}_{U}D f +\lambda \widehat{\nabla}_{U}V
\end{eqnarray}
and
\begin{equation} \label{6.5}
\widehat{\nabla}_{[U,V]}D f = -\widehat{Q}[U,V]+ \frac{1}{m}g([U,V],D f)D f+\lambda [U,V].
\end{equation}
Utilizing  (\ref{6.2})-(\ref{6.5}) and the relation $\widehat{R}(U, V)Df=\widehat{\nabla}_{U} \widehat{\nabla}_{V}Df-\widehat{\nabla}_{V} \widehat{\nabla}_{U}Df-\widehat{\nabla}_{[U, V]}Df$, we lead
\begin{eqnarray}
\widehat{R}(U,V)D f &=& (\widehat{\nabla}_{V}\widehat{Q})U-(\widehat{\nabla}_{U} \widehat{Q})V+\frac{\lambda}{m}\{ (V f)U-(U f)V\}\nonumber\\&&
+\frac{1}{m}\{ (U f)\widehat{Q} V-(V f)\widehat{Q} U \}.\nonumber
\end{eqnarray}
\end{proof}
Now contracting the equation (\ref{6.1}) over $U$, we obtain
\begin{equation}
\label{6.6}
\widehat{\mathcal{S}}(V, Df)=\frac{1}{2}(V \widehat{r})+\frac{2\lambda}{m}(V f)-\frac{1}{m}\{ (\frac{\widehat{r}}{2}+3)(V f)+(\frac{\widehat{r}}{2}-1)(\xi f)\mathfrak{\psi}(V)\}.
\end{equation}

Combining (\ref{6.6}) and (\ref{c5}), we have
\begin{eqnarray}
\label{6.7}
&&\frac{1}{2}(V \widehat{r})+\frac{2\lambda}{m}(V f)-\frac{1}{m}\{ (\frac{\widehat{r}}{2}+3)(V f)+(\frac{\widehat{r}}{2}-1)(\xi f)\mathfrak{\psi}(V)\}\nonumber\\&&
=(\frac{\widehat{r}}{2}+1)(V f)-(\frac{\widehat{r}}{2}-1)\mathfrak{\psi}(V)(\xi f).
\end{eqnarray}
Setting $V=\xi$ in (\ref{6.7}), we obtain
\begin{equation}
\label{6.8}
(2m+\widehat{r}-2\lambda+2)(\xi f)=0,
\end{equation}
since $\xi \widehat{r}=0$.

Now, from equation (\ref{6.1}) we have
\begin{equation}\label{6.10}
    g(\widehat{R}(U,V)\xi, Df)=(\frac{\lambda}{m}-\frac{2}{m})[\mathfrak{\psi}(V)(U f)-\mathfrak{\psi}(U)(V f)].
\end{equation}
Combining equation (\ref{c10}) and (\ref{6.10}), we find that
\begin{equation}\label{6.11}
\mathfrak{\psi}(V)(U f)-\mathfrak{\psi}(U)(V f)=(\frac{\lambda}{m}-\frac{2}{m})[\mathfrak{\psi}(V)(U f)-\mathfrak{\psi}(U)(V f)].
\end{equation}
Putting $V=\xi$ in the foregoing equation yields
\begin{equation}\label{6.12}
(\lambda-m-2)(U f)=0,
\end{equation}
where we have used $\xi f=0$.\par
Hence, either $(\lambda-m-2) = 0$ or $(\lambda-m-2) \ne 0$.\par
If $(\lambda-m-2) = 0$, then we get $\lambda=m+2=$ positive integer. Hence, the gradient $m$-quasi Einstein soliton is expanding.\par
If we suppose that $(\lambda-m-2) \ne 0$, then $(U f)=0$, which implies $f$ is a constant. Then, equation (\ref{6.1}) reveals that $N^3$ is an Einstein manifold. Hence, $N^3$ is of constant sectional curvature, since the manifold is of dimension $3$,.\par
Hence, we state:
\begin{theorem}
\label{thm6.2}
If the Riemannian metric of $N^3$ with a $SSNMC$ is a gradient $m$-quasi Einstein soliton, then either the soliton is expanding or it is a manifold of constant sectional curvature with respect to the $SSNMC$, provided $(2m+\widehat{r}-2\lambda+2) \ne 0$.
\end{theorem}
\section{acknowledgement}
We would like to thank the referee for reviewing the paper carefully and his or her valuable comments to improve the quality of the paper.

\bibliographystyle{kjmath}


\end{document}